\documentclass{amsart}

\usepackage{mathrsfs}
\usepackage{amscd}
\usepackage{amsmath}
\usepackage{amssymb}
\usepackage{amsthm}
\usepackage{epsf}
\usepackage{latexsym}
\usepackage{verbatim}
\usepackage[all, cmtip]{xy}
\usepackage{tikz}
\usetikzlibrary{positioning}
\usetikzlibrary{matrix}
\usepackage{float}
\usepackage{hyperref}
\usepackage{comment}
\usepackage{enumitem}
\usepackage{mathtools}
\usepackage{xcolor}
\usepackage{quiver}
\usepackage{thmtools}

\tikzstyle{bsq}=[rectangle, draw, thick, minimum width=.5cm, minimum height=.5cm]
\tikzstyle{bver}=[rectangle, draw, thick, minimum width=1cm, minimum height=2cm]
\tikzstyle{bhor}=[rectangle, draw, thick, minimum width=2cm, minimum height=1cm]

\usepackage[left=3.2cm, right=3.2cm]{geometry}

\newtheorem{theorem}{Theorem}[section]
\newtheorem{lemma}[theorem]{Lemma}

\newtheorem{corollary}[theorem]{Corollary}

\newtheorem{question}[theorem]{Question}
\newtheorem{varexample}[theorem]{Example}

\theoremstyle{definition}

\newtheorem{definition}[theorem]{Definition}

\def\Z{{\mathbb Z}}

\def\E{\mathcal{E}}

\def\scramble{\mathcal{S}}
\def\sep{\text{sep}}

\newcommand{\gon}{\operatorname{gon}}
\newcommand{\sn}{\operatorname{sn}}
\newcommand{\tw}{\operatorname{tw}}

\begin{document}
\title{The Scramble Number of Outerplanar Graphs}

\author[D. Rivera Laboy]{Doel Rivera Laboy}

\email{{\tt doel.riveralaboy@uky.edu}}
\date{}
\bibliographystyle{alpha}

\begin{abstract}
For planar graphs, it is known that their treewidth is bounded by $O(\sqrt{n})$, where $n$ is the number of vertices of the graph. A related invariant to treewidth, is the scramble number of graphs. Recently, Connor et. al proved that planar graphs of bounded maximal degree have scramble number bounded by $O(\sqrt{n})$. An open question is whether the scramble number of any planar graph follows this same bound. We give a definitive answer with an explicit bound for a subset of planar graphs, the simple outerplanar graphs and the simple near outerplanar graphs.

\end{abstract}

\maketitle

\section{Introduction}

\label{Sec:Intro}
Throughout, we consider all graphs to be simple. Treewidth is a graph parameter that was first introduced by Robertson and Seymour \cite{ROBERTSON}
in 1984. Roughly speaking, treewidth measures how close a graph is to a tree. It has become a parameter of interest in graph theory for two main reasons. Firstly, treewidth has
many algorithmic applications; for example, there are many results showing that NP-hard
problems can be solved in polynomial time on classes of graphs with bounded treewidth \cite{Bodlaender98}. Secondly, treewidth is closely related to other properties of graphs \cite{ParamsAndTreewidth}. It is closely related to the theories of balanced separators, minor-closed graphs, and chip firing.

A family of graphs whose treewidth has been extensively studied is the planar graphs. This has been primarily through the study of balanced separators and the separator number of planar graphs. In \cite{LiptonTarjan} Lipton and Tarjan give a proof that the separator number of a planar graph is bounded by $2\sqrt{2n}$, where $n$ is the number of vertices in the graph. Since then, several authors have improved the constant in the bound $c\sqrt{n}$. As a consequence, for planar graphs the treewidth is bounded by $O(\sqrt{n})$.

Treewidth is related to chip firing by being a lower bound for gonality of a graph \cite{deBruynGijswijt}. The gonality of a graph $G$, $\gon(G)$, corresponds to ``minimal" winning configurations in the chip firing game. Formally, gonality is the minimal degree of a rank $1$ divisor. However, treewidth is not the current best bound for the gonality of a graph. In \cite{HJJS}, the authors define the scramble number, $\sn(G)$ of a graph $G$. In that paper, they show that for any graph $G$, $\tw(G)\leq \sn(G)\leq \gon(G)$. Altogether, this leads to the question of how does scramble number behave for planar graphs. For some families of planar graphs such as trees, chains of cycles and the grid graph it was shown that the scramble number is equal to its treewidth \cite{HJJS}. For other planar graphs such as the strip graph \cite{jensen2024fibonaccisumsetsgonalitystrip} and the cube \cite{beougher2024chipfiringplatonicsolidsprimer}, the scramble number has been computed and it is bounded by $O(\sqrt{n}).$

In \cite{2026sizecomplexityscrambles} the authors proved that planar graphs of bounded degree have scramble number at most $O(\sqrt{n})$ and they also conjecture that for all planar graphs the scramble number is bounded by $O(\sqrt{n})$. Outerplanar graphs, the graphs of interest throughout this paper, are a subset of planar graphs which have treewidth at most $2$. As such, their result implies that for outerplanar graphs $sn(G) \leq 2\Delta$, where $\Delta$ denotes the highest degree of a vertex in $G$. However, there exist outerplanar graphs with $\Delta = O(n)$, such as the fan graph which has $\Delta = n-1$. In \cite{Hendrey18}, the gonality of the fan graph was computed and shown to be bounded by $O(\sqrt{n})$. However, the scramble number of the fan graph had not been computed. To this end, we compute the scramble number of the fan graph, later generalizing this approach to bound the scramble number of all outerplanar graphs. Doing this, we show that for outerplanar graphs, the scramble number is also bounded by $O(\sqrt{n})$. In particular, we give an explicit bound.
\begin{theorem}
    Let $G$ be an outerplanar graph on $n$ vertices, then $\sn(G)\leq (2+\sqrt{2})\sqrt{n}+4.$
\end{theorem}
We later compute the scramble number of the wheel graph, extending our bound on the scramble number to ``near outerplanar graphs".
\subsection*{Outline} In Section \ref{Sec:Background} we introduce all definitions and notation regarding outerplanar graphs, treewidth and scramble number. In Section \ref{Subsec:Fan graph} we compute the scramble number of the fan graph as a primer for bounding the scramble number of outerplanar graphs. Then, in Section \ref{Subsec:Outerplanar properties} we establish all necessary Lemmas and properties required to prove our main theorem, which is proven in Section \ref{Subsec:main theorem}. we prove that the scramble number. In Section \ref{Sec:Near Outerplanar} we define the notion of ``$k$-nearly outerplanar graphs", computing the scramble number of the wheel graph and extend our main result to ``$k$-nearly outerplanar graphs". Finally, in Section \ref{Sec: Open Questions} we state some future directions and open questions related to the work done in this paper.

\subsection*{Acknowledgments} Thanks to Dave Jensen for his support and insight, and to Ralph Morrison for helpful comments and feedback on earlier drafts.

\section{Background}
\label{Sec:Background}
\subsection{Outerplanar graphs}
Throughout, we assume the graphs are simple, connected, and have $n$ vertices. We denote by $\deg(v)$ the degree of a vertex $v$, that is, the number of edges connected to $v$. We now define the notions of a graph being planar or outerplanar.
\begin{definition}
    A graph is \emph{planar} if it may be drawn in the plane, such that no edges intersect. We call such a valid drawing an \emph{embedding} into the plane.
\end{definition}
Given an embedding of a planar graph, the connected regions enclosed by vertices and edges are called \emph{faces}. 
\begin{definition}
    A planar graph $G$ is \emph{outerplanar} if it may be embedded into the plane such that all of its vertices lie on the same unbounded face.
\end{definition}
We remark that for the graph to be planar or outerplanar, it suffices for there to exist an embedding. That is, it is possible to find other ways to draw the graph that don't meet the conditions in the definition. See Figure \ref{Fig:Outer} for examples of a graph drawn such that it is embedded either planarly or outerplanarly. 
\begin{figure}[h]
\begin{tikzpicture}[scale=2.0]

\draw [ball color=black] (-1,0) circle (0.55mm);
\draw [ball color=black] (-1,2) circle (0.55mm);
\draw [ball color=black] (0,1) circle (0.55mm);
\draw [ball color=black] (1,0) circle (0.55mm);
\draw [ball color=black] (1,2) circle (0.55mm);
\draw [ball color=black] (4,0) circle (0.55mm);
\draw [ball color=black] (4,2) circle (0.55mm);
\draw [ball color=black] (3,1) circle (0.55mm);
\draw [ball color=black] (6,0) circle (0.55mm);
\draw [ball color=black] (6,2) circle (0.55mm);

\draw (-1,0)--(-1,2);
\draw (-1,0)--(1,0);
\draw (1,0)--(1,2);
\draw (1,2)--(-1,2);
\draw (-1,0)--(0,1);
\draw (-1,2)--(0,1);

\draw (4,0)--(4,2);
\draw (4,0)--(6,0);
\draw (6,0)--(6,2);
\draw (6,2)--(4,2);
\draw (4,0)--(3,1);
\draw (4,2)--(3,1);

\end{tikzpicture}
\caption{A graph embedded planarly but not outerplanarly on the left. Then redrawn outerplanarly on the right.}
\label{Fig:Outer}
\end{figure}
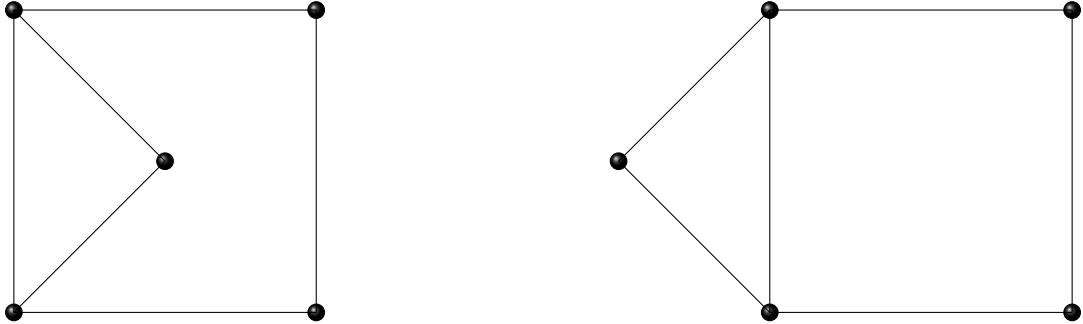

For a given graph $G$, a \emph{minor} is a subgraph that may be obtained via any sequence of the following operations: deleting vertices, deleting edges, or contracting edges. Planar graphs admit a characterization via excluded minors. That is, by Kuratowski's theorem, a finite graph is planar if and only if it does not contain $K_5$ or $K_{3,3}$ as a minor. In a similar manner, there is an excluded minor characterization for outerplanar graphs.
Outerplanar graphs were characterized first by Chartrand and Harary in \cite{chartrand1967planar}. 
\begin{theorem}[\cite{chartrand1967planar}]
    A graph $G$ is outerplanar if and only if it does not contain $K_4$ or $K_{2,3}$ as a minor.
\end{theorem}

\subsection{Treewidth and Separators}

It is known that a graph has treewidth at most $2$ if and only if it does not contain $K_4$ as a minor \cite{DUFFIN1965303}.
As a consequence, outerplanar graphs are a family of graphs with treewidth at most $2$. However, there exist graphs such as $K_{2,3}$ with treewidth $2$ that are not outerplanar.
Below we define treewidth in terms of brambles and bramble number, as it better relates to the scramble number definition seen in the next section.
\\
\begin{definition}
    A \emph{bramble} $\mathcal{B} = \{B_1, \ldots, B_m\}$ is a collection of connected subgraphs of $G$ that all pairwise ``touch". We say two subgraphs $B_i$ and $B_j$ ``touch" if either:
    \begin{enumerate}
        \item $B_i\cap B_j \neq \emptyset$
        \item There exists an edge in $G$ between a vertex in $B_i$ and a vertex in $B_j$.
    \end{enumerate}
\end{definition}
\begin{definition}
    A \emph{hitting set} for a bramble is a set of vertices of $G$ that has nonempty intersection with each of the subgraphs.
\end{definition}
The order of a bramble $|\mathcal{B}|$ is the smallest size of a hitting set. 
\begin{definition}
    The \emph{bramble number} of $G$, $\text{bn}(G)$, is the largest order of a bramble in $G$.
    \\
    The \emph{treewidth}, $\tw(G)$ is defined as one less than the bramble number of the graph. That is, $\tw(G) = \text{bn}(G)-1$.
\end{definition}

This definition via brambles is often used in the context of chip firing. However, in the setting of studying planar graphs, their treewidth has been primarily studied through the lens of balanced separators.
\begin{definition}
    A \emph{separator} is a set of vertices such that the removal of the vertices disconnects the graph into at least $2$ distinct non-empty components. We say a separator is \emph{balanced} if all disconnected components $C_i$ satisfy $|C_i|\leq \frac{2}{3}|V|$.
\end{definition}

\begin{definition}
    The separator number of a graph $G$, denoted $\sep(G)$, is the smallest $k$ such that for all subgraphs of $G$ there exists a balanced separator of size at most $k$.
\end{definition}

We now discuss the relation between separator number and treewidth. The following statement was remarked in \cite{chudnovsky2025coarsebalancedseparatorstreedecompositions} to follow from standard arguments.
\begin{lemma}\cite{chudnovsky2025coarsebalancedseparatorstreedecompositions}
    Let $G$ be any graph, $\sep(G) \leq \tw(G)+1 $. 
\end{lemma}

\begin{theorem}\cite{ROBERTSON}
     For every graph $G$, $\tw(G) \leq 4\sep(G)$.
\end{theorem}

\subsection{Scramble Number}
Treewidth is a lower bound on the graph invariant known as gonality \cite{deBruynGijswijt}. To improve upon treewidth as a bound for gonality, Harp, Jackson, Jensen, and Speeter defined the \emph{scramble number} of a graph $G$ in \cite{HJJS}. 

\begin{definition}
    A \emph{scramble} $\scramble = \{\mathcal{E}_1, \ldots, \mathcal{E}_m\}$ is a collection of connected subgraphs of $G$ called \emph{eggs}. 
\end{definition}

The order of a bramble was computed by finding hitting sets. However, for scrambles we are concerned with two operations, ``hitting" and ``cutting."  The hitting size of a scramble is defined the same way as the hitting size of a bramble.
\begin{definition}
    A \emph{hitting set} for a scramble is a set of vertices of G that has nonempty intersection with each of the eggs.
     The \emph{hitting size} $h(\scramble)$ of a scramble $\scramble$ is the minimum size of a hitting set for $\scramble$.
\end{definition}

\begin{definition}
    A \emph{cut} is a set of edges that disconnect a graph.
\end{definition}

Cuts are commonly denoted by the set of vertices on either side of the cut. Let $A$ and $B$ be disjoint vertex sets, we make use of the following notation
$$C(A,B) := \left\{e\in E(G) \mid  e=\{a,b\}, a\in A, b\in B\right\}.$$

\begin{definition}
    An \emph{egg-cut} $C$ for a scramble $\scramble$ is a cut such that at least one $\mathcal{E} \in \scramble$ lies completely on each side of $C$.
    The \emph{egg-cut size} $e(\scramble)$ for a scramble $\scramble$ is the size of the smallest egg-cut for $\scramble$.
\end{definition}

Altogether, we can define the order of a scramble $\scramble$.

\begin{definition}
    The \emph{order} of a scramble $\scramble$ is $\|{\scramble}\| = \min\{h(\scramble), e(\scramble)\}$.
\end{definition}

Finally, we define the scramble number of a graph $G$.

\begin{definition}
    The \emph{scramble number} $\sn(G)$ of a graph $G$ is $$\sn(G) = \max\{\|{\scramble}\| : \scramble \text{ is a scramble on } G\}.$$
\end{definition}

\begin{lemma}[\cite{HJJS}{Theorem 1.1}]
\label{lem:sn <= gon}
    For any graph $G$, $\tw(G)\leq \sn(G) \leq \gon(G)$.
\end{lemma}

\section{Scramble number of outerplanar graphs}
\subsection{Scramble number of the fan graph}
\label{Subsec:Fan graph}
 The \emph{fan graph} $\mathcal{F}_m$ is the graph with $m+1$ vertices $v_0, \ldots, v_m$, and edges between $v_0$ and $v_i$ for all $i \geq 1$, and between $v_i$ and $v_j$ if $\vert i-j \vert = 1$ for all $i,j \geq 1$. See Figure~\ref{Fig:Fan} for an example of the fan graph and an example of a scramble on the graph.

\begin{figure}[h]
\begin{tikzpicture}[scale=2.0]
\filldraw[color=red, fill=white!5, very thick](0.5,0) circle (0.75);
\filldraw[color=magenta, fill=white!5, very thick](2.5,0) circle (0.75);
\filldraw[color=cyan, fill=white!5, very thick](4.5,0) circle (0.75);
\filldraw[color=olive, fill=white!5, very thick](2.5,2) circle (0.5);

\draw [ball color=black] (0,0) circle (0.55mm);
\draw [ball color=black] (1,0) circle (0.55mm);
\draw [ball color=black] (2,0) circle (0.55mm);
\draw [ball color=black] (3,0) circle (0.55mm);
\draw [ball color=black] (4,0) circle (0.55mm);
\draw [ball color=black] (5,0) circle (0.55mm);
\draw [ball color=black] (2.5,2) circle (0.55mm);
\draw (0,0)--(1,0);
\draw (1,0)--(2,0);
\draw (2,0)--(3,0);
\draw (3,0)--(4,0);
\draw (4,0)--(5,0);
\draw (2.5,2)--(0,0);
\draw (2.5,2)--(1,0);
\draw (2.5,2)--(2,0);
\draw (2.5,2)--(3,0);
\draw (2.5,2)--(4,0);
\draw (2.5,2)--(5,0);
\draw (2.5,2.25) node {{\small $v_0$}};
\draw (0,-0.25) node {{\small $v_1$}};
\draw (1,-0.25) node {{\small $v_2$}};
\draw (2,-0.25) node {{\small $v_3$}};
\draw (3,-0.25) node {{\small $v_4$}};
\draw (4,-0.25) node {{\small $v_5$}};
\draw (5,-0.25) node {{\small $v_6$}};

\end{tikzpicture}
\caption{The fan graph $\mathcal{F}_6$. Depicted is a scramble with four eggs $S =\{\{ v_0\},\{ v_1, v_2\},\{ v_3, v_4\},\{ v_5, v_6\}\}$. This is a scramble with $h(S) = 4$ and $e(S) = 3$.}
\label{Fig:Fan}
\end{figure}
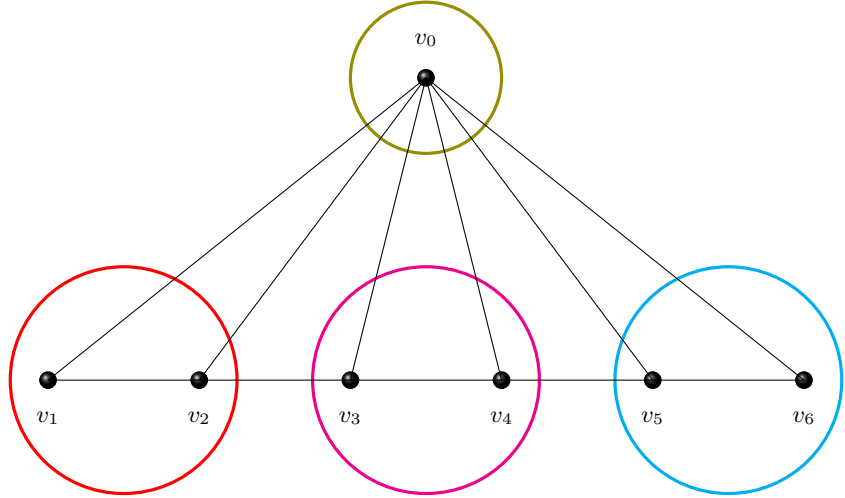  

We now compute the scramble number of the fan graph. The formula will depend on how close $m$ is to $\lfloor \sqrt{m}\rfloor.$ We do this by representing $m$ uniquely as $k^2+c$ with $k$ a positive integer and $0\leq c\leq 2k$. The approach used to compute the scramble number of the fan graph serves as a primer for bounding the scramble number of outerplanar graphs.
\begin{lemma}\label{lem: fan gon}
    Let $F_m$ be the fan graph on $m+1$ vertices. If $m = k^2+c$ for some $k\in \Z_{>0}$ with $0\leq c\leq 2k$, then
    $$\sn(F_m) = \begin{cases}
        \left\lfloor\sqrt{m} \right\rfloor+1 & \text{if } c\leq k+1\\ 
        \left\lfloor\sqrt{m} \right\rfloor+2 & \text{otherwise }\\
    \end{cases}.$$
\end{lemma}
\begin{proof}
    We proceed by cases. For each case, we produce a scramble which yields the appropriate lower bound.
    \begin{enumerate}
        \item If $0\leq c\leq k+1$, define the scramble $\scramble =\{\E_0, \ldots, \E_k \}$ where $\E_0 = \{v_0\}$ and $\E_i = \{v_j\mid (i-1)k+1\leq j \leq ik\}$ for $1\leq i \leq k$. We note that we have selected $k+1$ many disjoint eggs. As such $h(S) = k+1$. Since $v_0$ is connected to every other vertex, any egg-cut requires that we disconnect an egg from the fan vertex, $v_0$. As such, let $C$ be any connected component of $G$ that does not contain $v_0$. We must at least remove the $|C|$ edges coming from $v_0$ to $C$. If $C$ is the entire path, then we remove $|C| = m$ many edges. If $C$ contains at most one vertex of degree $2$, then we must remove an additional edge to disconnect $C$ from the remaining path. Hence any cut requires at least $\min\{m, |C|+1\}$ edges. Altogether, we find that a minimal egg cut is given by removing the $k+1$ many edges coming out of the egg $\E_1$. As such, $e(S) = k+1$ and $\|\scramble\| = k+1$.

        \item If $k+2 \leq c \leq 2k$ define $\scramble =\{\E_0, \ldots, \E_{k+1} \}$ where $\E_0 = \{v_0\}$, $\E_1 = \{v_j \mid 1\leq j \leq k+1 \} $ $\E_{k+1} = \{v_j\mid k^2+2\leq j \leq k^2+k+2\}$  and $\E_i = \{v_j\mid (i-1)k+2\leq j \leq ik+1\}$ for $2\leq i \leq k$. 
        
        We note that we have selected $k+2$ many disjoint eggs. As such $h(S) = k+2$. We then note that all eggs require at least $k+2$ edges to disconnect them from the rest of the graph. Hence, $e(S) = k+2$ and $\|\scramble\| = k+2$. 
    \end{enumerate}
    Now, for each case, we shall consider an arbitrary scramble $\scramble$ and conclude that its order must be bounded.
    \begin{enumerate}
        \item If $0\leq c\leq k+1$, define
    $H =\{v_0 \}\cup\{v_{jk+1} \mid 1\leq j\leq k-1 \}$.
    If $H$ is a hitting set, then $\|\scramble\|\leq k$. Otherwise, we note that $H$ separates $F_m$ into disjoint path components $P_1, \ldots, P_k$. We now have the following cases:
    \begin{enumerate}
        \item If only one unhit egg remains or all remaining unhit eggs pairwise intersect with each other, then we may add one vertex to $H$ to make it a hitting set. Then $\|\scramble\|\leq h(\scramble) \leq |H|+1\leq k+1$
        \item If two unhit eggs are disjoint and lie in distinct path components, then at least one of the path components may be disconnected from the rest of the graph by the removal of at most $|H|+1$ many edges. This produces an egg cut and, as such, $\|\scramble\| \leq e(\scramble) \leq k+1$.
        \item If two unhit eggs are disjoint and both lie in the same path, we have two possibilities. If $c<k+1$, then by the pigeonhole principle, at least one of the eggs is of size less than or equal to $|H|-1$. As such, at least one egg may be disconnected from the rest of the graph by the removal of at most $|H|+1$ many edges.

        If $c=k+1$, the ``last" path component has $2k$ vertices, allowing for the possibility of two disjoint eggs of size $k$. However, in this scenario, one of the eggs must have a vertex of degree $2$, allowing for a cut of size $k+1$.
        Hence, $\|\scramble\| \leq e(\scramble)\leq k+1$.
    \end{enumerate}

    \item If $k+2\leq c\leq 2k$, then we repeat the argument above with $H =\{v_0\}\cup \{v_{jk+1} \mid 1\leq j\leq k \}$ and cases wether $c<2k$ or $c=2k$, to conclude that $\|\scramble\|\leq k+2$.
    \end{enumerate}

    Altogether, $$\sn(F_m) = \begin{cases}
        \left\lfloor\sqrt{m} \right\rfloor+1 & \text{if } c\leq k+1\\ 
        \left\lfloor\sqrt{m} \right\rfloor+2 & \text{otherwise }\\
    \end{cases}.$$
    
\end{proof}

Lemma \ref{lem: fan gon} implies that there exist outerplanar graphs with treewidth $2$ and scramble number $O(\sqrt{n})$. Additionally, in \cite{Hendrey18}, Hendrey showed that the fan graph has gonality roughly $2\lfloor n\rfloor$. This implies that the fan graphs are a family of graphs for which there is a gap between the scramble number and the gonality.  

\subsection{Properties of outerplanar graphs}
\label{Subsec:Outerplanar properties}
Throughout this section we lay out all of the necessary pieces to prove Theorem \ref{thm: main}. First, we remark that outerplanar graphs have a collection of edges which lie on the outside face. Each edge that does not lie on the outside face defines a separator (not necessarily balanced) for the graph, namely its two endpoints.

In order to have additional structure, we restrict to the case of maximal outerplanar graphs. That is, outerplanar graphs where the addition of any edge causes the graph to not be outerplanar. Since the scramble number is subgraph monotone \cite{HJJS}, it suffices to study only the maximal case. Maximal outerplanar graphs are known to be sparse, having exactly $2n-3$ edges. However, as shown in \cite{Bonnet}, there exist sparse graphs with linearly large treewidth and consequently linearly large scramble number. As such, beyond being sparse, we also make use of the fact that we can bound the number of high degree vertices.

\begin{lemma}\label{lem:sparse outerplanar}[ {\cite{DegsOuter}} Theorem 3.4]
    Let $G$ be an outerplanar graph with
    $$A_k = \{v\mid \deg(v)\geq k\}.$$
    If $k\geq 6$ then
    $$|A_k|\leq \left\lfloor \frac{n-6}{k-4}\right\rfloor.$$
\end{lemma}

As a corollary, we have that
\begin{corollary}\label{coll: vertex deg pre}
 Let $G$ be an outerplanar graph with
    $A_k = \{v\mid \deg(v)\geq k\}$ and $c>0$.
    If $n\geq \frac{4}{c^2}$ then
    $$|A_{ c\sqrt{n}+4}| < \frac{\sqrt{n}}{c}.$$
\end{corollary}
 \begin{proof}
     If $n\geq \frac{4}{c^2}$, we have that $c\sqrt{n}+4\geq 6$. Altogether, by Lemma \ref{lem:sparse outerplanar}
     $$|A_{c\sqrt{n}+4}|\leq  \left\lfloor \frac{n-6}{c\sqrt{n}} \right\rfloor =\left\lfloor \frac{\sqrt{n}}{c}-  \frac{6}{ c\sqrt{n}} \right\rfloor <\frac{\sqrt{n}}{c}.$$
 \end{proof}
We remark that the choice of constant $c$ is determined to optimize the number of edges in the cuts we will produce in Lemma \ref{lem: Fan Process}. In particular, the optimal value is given by minimizing the expression $c+\frac{1}{c}$. As such, we make use of $c = 1$.
\begin{corollary}\label{coll: vertex deg}
 Let $G$ be an outerplanar graph with
    $A_k = \{v\mid \deg(v)\geq k\}.$
    If $n\geq 4$ then
    $$\left|A_{ \sqrt{n}+4}\right| < \sqrt{n}.$$
\end{corollary}

It will also be useful to use the ``local structure" of outerplanar graphs. To this end, we define the following:
\begin{definition}
    We denote the neighborhood of a vertex $v$ as follows:
    $$N(v):=\{ w\in V(G) \mid (v,w)\in E(G) \}.$$
\end{definition}

\begin{definition}
    Let $A\subseteq V(G)$, we denote by $G[A]$ the subgraph induced by $A$.
\end{definition}
Outerplanar graphs have additional structure which allows us to understand the graph induced by any vertex an its neighborhood. 

\begin{lemma}\label{lem: neigh is subfan}
    If $G$ is an outerplanar graph, then $G[N(v)\cup\{v\}]$ is isomorphic to a subgraph of $F_{\deg(v)}$.
\end{lemma}
\begin{proof}
    We map the vertices in $G[N(v)\cup\{v\}]$ to a subgraph of the fan graph with vertices $u_0$ through $u_{k}$ with $k = \deg(v)$. Note that $G$ outerplanar implies that $G[N(v)\cup\{v\}]$ is also outerplanar. We first show that all vertices in $N(v)$ have degree at most $3$ in $G[N(v)\cup\{v\}]$. For contradiction, assume that some $v_i$ is adjacent to $v, v_a, v_b,$ and $v_c$ where $v_a, v_b, v_c\in N(v)$. This leads to the $K_{2,3}$ minor pictured below, contradicting outerplanarity.
    \\
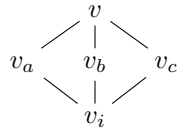
\[\begin{tikzcd}[cramped,sep=small]
	& v & \\
	{v_a} & {v_b} & {v_c} \\
	& {v_i}
	\arrow[no head, from=1-2, to=2-3]
	\arrow[no head, from=2-1, to=1-2]
	\arrow[no head, from=2-2, to=1-2]
	\arrow[no head, from=3-2, to=2-1]
	\arrow[no head, from=3-2, to=2-2]
	\arrow[no head, from=3-2, to=2-3]
\end{tikzcd}\]
    \\
    Every outerplanar graph has at least two vertices with degree strictly less than $3$. This then implies that all vertices in $G[N(v)]$ have degree at most $2$, and at least two vertices of degree at most $1$. As a consequence, $G[N(v)]$ is the union of paths $P^1, \ldots, P^m$ for some $1\leq m \leq k$. Altogether, we may conclude that $G[N(v)\cup \{v\}]$ is isomorphic to the subgraph of $F_k$ obtained by removing the edges that connect $P^j$ to $P^{j+1}$ for $1\leq j\leq m-1$.
    \end{proof}

\begin{corollary}\label{rem: fan}
    If $G$ is a maximal outerplanar graph, then $G[N(v)\cup\{v\}]$ is isomorphic to $F_{\deg(v)}$.
\end{corollary}
\begin{proof}
    Since $G$ is maximal outerplanar, all vertices may be embedded as lying on an exterior face. We may then order all vertices such that $v$ is the first vertex $v_0$ and each $v_i$ is adjacent to $v_{i+1}$. 

    By Lemma \ref{lem: neigh is subfan}, $G[N(v)\cup\{v\}]$ is isomorphic to a subgraph of $F_{\deg(v)}$. Once all vertices are embedded to a cycle as the exterior face, if all neighbors of $v_0$ have consecutive indices, then $G[N(v)\cup\{v\}]$ is isomorphic to a fan. Otherwise, there exists $i$ and $j$ such that $v_i\in N(v)$, $i+1<j$, and $j$ is the minimal integer greater than $i$ such that $v_j\in N(v)$. For contradiction, assume that $v_i$ is not adjacent $v_j$. This implies that there exists some $1\leq \ell \leq j-i-1$ such that $v_i\to  v_{i+1} \to \ldots \to v_{i+\ell} \to v_j$ forms a cycle with no chords, contradicting the maximality of $G$. Altogether, for all $v_i\in N(v)$, $v_i$ is adjacent to $v_j$ where $j$ is the minimal integer greater than $j$ such that $v_j\in N(v)$. Therefore, $G[N(v)\cup\{v\}]$ is isomorphic to $F_{\deg(v)}$. 
\end{proof}

Throughout the remainder of this section we make a distinction between the labeling of the vertices inherited from the graph and the labeling inherited by the isomorphism to the fan graph. We make use of $u_j$ to denote a vertex $v_i$, labeled under the isomorphism to a particular fan graph.

\begin{lemma}\label{lem: separate outer}
    Let $G$ be a maximal outerplanar graph. Under the isomorphism that maps $G[N(v)\cup\{v\}]$ to $F_{\deg(v)}$, let $x,y\in N(v)$ map to non adjacent path vertices $u_i$ and $u_j$ such that $i<j$. We may produce a cut $C$ such that
    \begin{enumerate}
        \item $C$ disconnects at least two non-empty vertex sets,
        \item the vertices $u_{i+1}, \ldots, u_{j-1}$ lie in the same component, and $v$ lies in another component,
        \item $|C|\leq \frac{1}{2}\deg(x)+\frac{1}{2}\deg(y)+(j-i)$.
    \end{enumerate}
\end{lemma}
\begin{proof}
    Recall that $\{v,x\}$ and $\{v,y\}$ are separators. This implies that $v$ and $x$ partition $G$ into two subgraphs $A_x$ and $B_x$ such that there are no paths between $A_x$ and $B_x$. Similarly, we may define $A_y$ and $B_y$. Without loss of generality $|A_x|\leq |B_x|$ and $|A_y|\leq |B_y|$. Let $$ C = \{(v,u_{\ell})\mid i+1 \leq \ell \leq j-1 \} \cup \{(a, x) \mid a\in A_x\} \cup \{(a, y) \mid a\in A_y\}.$$
    \\
    For contradiction, assume that there exists a path in $G\setminus C$ from $u_k$ to $v$ where $i<k<j$. We have removed the edge $(v,u_k)$, hence the path contains some vertex $z$ with $z\notin \{v,u_i,u_j,u_k\}$. Then we form a $K_{2,3}$ minor with the vertices $v, u_i, u_k, u_j, z$ as pictured below.
    \\
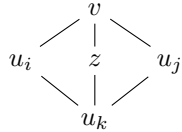
\[\begin{tikzcd}[cramped,sep=small]
	& v & \\
	{u_i} & z & {u_j} \\
	& {u_k}
	\arrow[no head, from=1-2, to=2-3]
	\arrow[no head, from=2-1, to=1-2]
	\arrow[no head, from=2-2, to=1-2]
	\arrow[no head, from=3-2, to=2-1]
	\arrow[no head, from=3-2, to=2-2]
	\arrow[no head, from=3-2, to=2-3]
\end{tikzcd}\]
    \\
    This contradicts that $G$ is outerplanar. Therefore, the edges removed in $C$ define a cut with $|C|\leq \deg(x)+\deg(y)+(j-i)$. By construction, one connected component contains $v$ and another component contains the path vertices $u_{i+1}$ through $u_{j-1}$.
\end{proof}
In order to prove our main theorem we will make use of a similar process of the proof of Lemma \ref{lem: fan gon}, where we generate a set of vertices that serves as either a hitting set or forces unhit eggs to be separated by a bounded amount of edges. However, we will need to make a choice of vertices that generalizes to all possible maximal outerplanar graphs. In the next proof we take a set similar to the one in Lemma \ref{lem: fan gon} and generate a new set which satisfies the necessary properties.
\begin{lemma}\label{lem: Fan Process}
    Let $G$ be a maximal outerplanar graph on $n$ vertices with $n\geq 4$ and $\left( v_i,v_j\right)$ an interior edge. If $\deg\left( v_i\right)>\sqrt{2n}$ and $\deg\left( v_i\right)\geq \deg\left( v_j\right)$, then there exists a set $\mathcal{H}$ that satisfies the following:
    \begin{enumerate}
        \item $v_i,v_j\in \mathcal{H}$.
        \item $G\setminus \mathcal{H}$ is separated into components $C_1, \ldots, C_m$ such that $|C(C_i, G\setminus C_i)|\leq  (2+\sqrt{2})\sqrt{n}+4$. That is, the outdegree of each component is at most $(2+\sqrt{2})\sqrt{n}+4.$ 
        \item $|\mathcal{H}|\leq 2\left\lfloor \frac{\Delta}{\sqrt{2n}}\right\rfloor$.
    \end{enumerate}
\end{lemma}
\begin{proof}
Let $k = \deg\left( v_i\right)$. By Corollary \ref{rem: fan}, $G[N\left( v_i\right)\cup\{v_i\}] \cong F_{k}$ with $v_i$ mapping to $u_0$ and the neighboring vertices mapping to $u_1$ through $u_k$. 
By Corollary \ref{coll: vertex deg}, $G$ has at most $\sqrt{n}$ vertices of degree greater than or equal to $\sqrt{n}+4$. This implies that for each $1\leq j \leq \left\lfloor\frac{k}{\sqrt{2n}}\right\rfloor$ there exists an integer $-\frac{\sqrt{n}}{2}\leq \ell_j\leq \frac{\sqrt{n}}{2}$ such that $\deg\left( u_{\left\lfloor j\sqrt{2n} \right\rfloor+\ell_j}\right)<\sqrt{n}$.

Thus, we define $H =\{u_0\}\cup \{u_{\left\lfloor j\sqrt{2n} \right\rfloor+\ell_j}\mid 1\leq j \leq \left\lfloor\frac{k}{\sqrt{2n}} \right\rfloor \}$. We note that by definition $|H| = \left\lfloor\frac{k}{\sqrt{2n}}\right\rfloor \leq \left\lfloor\frac{\Delta}{\sqrt{2n}}\right\rfloor $. 

Recall that the endpoints of any interior edges define a separator. As such, we look at the possible components separated by edges with vertices in $H$. 
For each $j$, by Lemma \ref{lem: separate outer}, there exists a cut $C_j$ with $u_{\left\lfloor j\sqrt{2n} \right\rfloor+\ell_j+1}$ through $u_{\left\lfloor \left( j+1\right)\sqrt{2n} \right\rfloor+\ell_{j+1}-1}$ in one component and $u_0$ in the other, such that
\begin{align*}
    |C_j|&\leq \frac{1}{2} \deg\left( u_{\left\lfloor j\sqrt{2n} \right\rfloor+\ell_j}\right)+\frac{1}{2}\deg\left( u_{\left\lfloor \left( j+1\right)\sqrt{2n} \right\rfloor+\ell_{j+1}}\right)+\left( \left\lfloor \left( j+1\right)\sqrt{2n} \right\rfloor+\ell_{j+1}-\left\lfloor j\sqrt{2n} \right\rfloor-\ell_j\right)\\
    &< \frac{1}{2}(\sqrt{n}+4)+\frac{1}{2}(\sqrt{n}+4)+\left( \sqrt{2n}+\ell_{j+1}-\ell_j\right)\\
    &\leq  \frac{1}{2}(\sqrt{n}+4)+\frac{1}{2}(\sqrt{n}+4)+\left( \sqrt{2n}+\sqrt{n}\right)\\
    &= (2+\sqrt{2})\sqrt{n}+4.
\end{align*}

We must also consider the ``endpoints" of the fan, as $\left( u_0,u_{\lfloor\sqrt{2n}\rfloor+\ell_{1}}\right)$ and $\left( u_0,u_{\left\lfloor\left\lfloor\frac{k}{\sqrt{2n}}\right\rfloor\sqrt{n}\right\rfloor+\ell_{\left\lfloor\frac{k}{\sqrt{2n}}\right\rfloor}}\right)$ are separators. We define the following respective cuts
$$C_0 = \{\left( u_0,u_m\right) \mid 1 \leq m \leq {\lfloor\sqrt{2n}\rfloor+\ell_{1}-1} \} \cup N\left( u_{\lfloor\sqrt{2n}\rfloor+\ell_{1}}\right) \text{ and }$$
$$C_{\left\lfloor\frac{k}{\sqrt{2n}}\right\rfloor} = \left\{\left( u_0,u_m\right) \mid {\left\lfloor\left\lfloor\frac{k}{\sqrt{2n}}\right\rfloor\sqrt{n}\right\rfloor+\ell_{\left\lfloor\frac{k}{\sqrt{2n}}\right\rfloor}-1} \leq m \leq k \right\}\cup N\left( u_{\left\lfloor\left\lfloor\frac{k}{\sqrt{2n}}\right\rfloor\sqrt{n}\right\rfloor+\ell_{\left\lfloor\frac{k}{\sqrt{2n}}\right\rfloor}}\right).$$

Now we repeat the above with $v_j$ to generate an $H'$ such that $|H'|\leq \left\lfloor \frac{\Delta}{\sqrt{2n}} \right\rfloor$ and all components obtained by removing an adjacent pair of vertices in $H'$ have outdegree bounded by $(2+\sqrt{2})\sqrt{n}+4$. 
Altogether, we set $\mathcal{H} = H\cup H'$ where 
$$|\mathcal{H}|\leq |H|+|H'|\leq 2  \left\lfloor \frac{\Delta}{\sqrt{2n}} \right\rfloor ,$$
and all components separated by $\mathcal{H}$, have a cut that disconnects them with at most $(2+\sqrt{2})\sqrt{n}+4$ edges.
\end{proof}

\subsection{Proof of main theorem}
\label{Subsec:main theorem}
We note that the lemmas proven so far make use of an assumption of $n$ meeting some inequality. As such, we take care separately of the cases of outerplanar graphs with fewer vertices. To do this, we make use of graph colorings and independent sets. 
\begin{definition}
    A set of vertices is \emph{independent} if no two vertices share an edge with each other. The \emph{independence number} of a graph, $\alpha(G)$, is the size of the largest possible independent set in $G$.
\end{definition}
\begin{definition}
    A \emph{vertex coloring} is a map that assigns a color to each vertex such that no two adjacent vertices have the same color. The \emph{chromatic number} of a graph, $\chi(G)$, is the smallest number of colors required to give $G$ a vertex coloring.
\end{definition}
As such, a valid vertex coloring yields collections of independent sets. We may then conclude that $\alpha(G) \geq \frac{n}{\chi (G)}$. Then we make use of the fact that gonality is an upper bound on scramble number.
\begin{lemma}[\cite{deveau2016gonality}*{Proposition 3.1}]\label{lem: gon bound}
    Let $G$ be a simple connected graph, $\gon(G)\leq |V|-\alpha(G)$, where $\alpha(G)$ is the independence number of $G$.
\end{lemma}
Altogether, if we have a bound on the chromatic number of the graph, we may bound the scramble number. For both planar graphs and outerplanar graphs, such bounds on the chromatic number are known. Planar graphs have a chromatic number of at most $4$ \cite{ROBERTSON19972} and outerplanar graphs have chromatic number at most $3$ \cite{doi:10.1137/0607016}.

\begin{lemma}\label{lem:sn outer when small v3}
    Let $G$ be a simple graph on $n$ vertices. If
    \begin{enumerate}
        \item $G$ is planar with $n\leq 20$ or
        \item $G$ outerplanar with $n\leq 26$,
    \end{enumerate}
    then
    $\sn(G)\leq \gon(G)\leq (2+\sqrt{2})\sqrt{n}.$
\end{lemma}
\begin{proof}
    Since $G$ is simple, we have that $\gon(G)\leq n-\alpha(G)$. Consequently
    $$\sn(G) \leq \gon(G)\leq n-\alpha(G)\leq n-\frac{n}{\chi(G)}.$$
    We now proceed by cases:
    \begin{enumerate}
        \item If $n\leq 20$ and $G$ is planar, this implies that $\frac{3n}{4}\leq (2+\sqrt{2})\sqrt{n}$. Since $G$ is planar, then it is four colorable and $\chi(G)\leq 4$. Altogether,
    $$\sn(G) \leq \gon(G) \leq n-\frac{n}{4}= \frac{3n}{4}\leq (2+\sqrt{2})\sqrt{n}.$$
        \item If $n\leq 26$ and $G$ is outerplanar, this implies that $\frac{2n}{3}\leq (2+\sqrt{2})\sqrt{n}$. Since $G$ is outerplanar, then it is three colorable and $\chi(G)\leq 3$. Altogether,
    $$\sn(G)\leq \gon(G) \leq n-\frac{n}{3}= \frac{2n}{3}\leq (2+\sqrt{2})\sqrt{n}.$$
    \end{enumerate}

\end{proof}

\begin{theorem}\label{thm: main}
    Let $G$ be an outerplanar graph on $n$ vertices. $$\sn(G)\leq  (2+\sqrt{2})\sqrt{n}+4.$$
\end{theorem}
\begin{proof}
    By Lemma \ref{lem:sn outer when small v3}, if $n\leq 26$ then the bound is met. Throughout the remainder of the proof, we assume $n> 26$. Let $S$ be an arbitrary scramble on $G$ a maximal outerplanar graph. If $e(S) = \infty$, then $S$ is a bramble, hence by $G$ outerplanar we have $h(S)\leq \tw(G)\leq  2$.

    If $e(S)<\infty$, then we take the following approach. We define three subsets of vertices $\mathcal{V}$, $H_1$ and $H_2$. Initially, $\mathcal{V}=V(G)$ and $H_1 = H_2 = \emptyset$. The following algorithm terminates if  any of the following occur:
    \begin{itemize}
        \item $ H_1\cup H_2$ is a hitting set for $S$.
        \item We find an egg cut of size $\leq (2+\sqrt{2})\sqrt{n}+4$.
        \item $|\mathcal{V}|\leq \sqrt{n}$.
    \end{itemize}
    At each step of the algorithm, we make use of balanced separators to search through a sequence of nested subgraphs of $G$ for either a hitting set for $S$ or an egg cut of size at most $(2+\sqrt{2})\sqrt{n}+4$.
    Throughout the algorithm $\mathcal{V}$ will keep track of the vertices in the subgraph of $G$ which we are searching in.

    We start each iteration by finding a balanced separator $\{v_i, v_j\}$ for $G[\mathcal{V}]$. Let $A$ and $B$ be the components separated by the removal of $\{v_i, v_j\}$. We say two eggs $\mathcal{E}_1$ and $\mathcal{E}_2$ are separated by $\{v_i, v_j\}$ if $\mathcal{E}_1\subseteq A$ and $\mathcal{E}_2\subseteq B$. We say an egg $\mathcal{E}$ is \emph{unhit} if $\mathcal{E}\cap (\{v_i,v_j\}\cup H_1\cup H_2) = \emptyset$. We now proceed by cases depending on wether $H_1$ is nonempty or not.

    If $H_1 = \emptyset$ and $\{v_i,v_j\}$ does not separate unhit eggs, then all unhit eggs lie in one of the two components $A$ or $B$. Without loss of generality all unhit eggs lie in component $B$ with $|B|\leq \frac{2}{3}n$. We have two cases:
    \begin{enumerate}
        \item If $\{v_i,v_j\}$ do not hit any eggs, then we set $\mathcal{V} = B$ and we repeat the process.

        \item Otherwise, we set $H_1 = \{v_i,v_j\}$, $\mathcal{V} = B$ and we repeat the process.
    \end{enumerate}

    If $H_1=\emptyset$ and $\{v_i, v_j\}$ separates unhit eggs, then we have two cases:
        \begin{enumerate}
            \item If $\deg_{A}(v_i)+\deg_{A}(v_j)\leq 2\sqrt{2n}$ or $\deg_{B}(v_i)+\deg_{B}(v_i)\leq 2\sqrt{2n}$, then we define the cuts $C_A = \{(v_i,w)\in E(G)\mid w\in A\}\cup \{(v_j,w)\in E(G)\mid w\in A\}$ and $C_B = \{(v_i,w)\in E(G)\mid w\in B\}\cup \{(v_j,w)\in E(G)\mid w\in B\}$. Consequently, either $C_A$ or $C_B$ produces an eggcut with size at most $2\sqrt{2n}$.
            \item Otherwise, without loss of generality $\deg(v_i)>\sqrt{2n}$ with $\deg(v_i)\geq \deg(v_j)$.
            \\
            By Lemma \ref{lem: Fan Process} there exists a subset of vertices $\mathcal{H}$ with $|\mathcal{H|}\leq 2\left\lfloor\frac{\Delta}{\sqrt{2n}}\right\rfloor$ such that $v_i,v_j\in \mathcal{H}$ and $\mathcal{H} $ separates the graph into components that can be separated with the removal of at most $(2+\sqrt{2})\sqrt{n}+4$ edges. We then set $H_1 = \mathcal{H}$ and either
            \begin{itemize}
                \item $H_1$ is a hitting set. 
                \item Two eggs unhit by $H_1$ are separated, hence by construction there exists a cut $C$ with $|C|\leq (2+\sqrt{2})\sqrt{n}+4$, or
                \item no two unhit eggs are separated. Hence, all unhit eggs are contained in some component $B'$.
                In this case, we set $\mathcal{V} = B'$ and repeat the process above on the separated component $B'$.
            \end{itemize}
    \end{enumerate}
    If $H_1\neq \emptyset$, then we consider updating $H_2$. If either $|C_A|\leq 2\sqrt{2n}$ or $|C_B|\leq 2\sqrt{2n}$, we set $H_2 = \{v_i,v_j\}$. Otherwise, by Lemma \ref{lem: Fan Process} there exists a subset of vertices $\mathcal{H}$ with $|\mathcal{H|}\leq 2\left\lfloor\frac{\Delta}{\sqrt{2n}}\right\rfloor$ such that $v_i,v_j\in \mathcal{H}$ and $\mathcal{H}$ separates the graph into components that can be separated with at most $(2+\sqrt{2})\sqrt{n}+4$ many edges.
    We then set $H_2 = \mathcal{H}$. In either case, there are three scenarios:

    \begin{enumerate}
                \item $H_1\cup H_2$ is a hitting set.  
                \item Two eggs unhit by $H_2$ are separated, hence we can find a cut $C$ with $|C|\leq (2+\sqrt{2})\sqrt{n}+4$, or
                \item no two unhit eggs are separated. Hence, all unhit eggs are contained in a component $B'$. We set $\mathcal{V} = B'$ and in this case either:                \begin{enumerate}
                    \item $H_2$ hits all of the same eggs as $H_1$. In which case, we continue the process with $H_2$ as our new $H_1$. That is, our new $H_1$ is $H_2$ and our new $H_2$ is empty. 
                    \item There exist eggs hit by $H_1$ but not $H_2$. If any of these lie in $B'^C$ then we may find an egg cut of size $\leq (2+\sqrt{2})\sqrt{n}+4$. Otherwise, all eggs hit by $H_1$ but not $H_2$ lie in $B'$. As such, $H_1$ was again redundant. That is, we continue the process with $H_2$ as our new $H_1$.   
                    \end{enumerate}
        \end{enumerate}
    Recall that a condition for the algorithm terminating is based on the size of $\mathcal{V}$. If $|\mathcal{V}|\leq \sqrt{n}$, then $G[\mathcal{V}]$ has at most $2\sqrt{n}+3$ edges. This implies that any two disjoint unhit eggs in $G[\mathcal{V}]$ may be cut with the removal of at most $2\sqrt{n}+3$ edges. Otherwise, if all remaining unhit eggs pairwise intersect, then there exists a vertex $w$ such that $w\cup H_1 \cup H_2$ is a hitting set. Altogether, the algorithm will terminate in at most $\log_{\frac{3}{2}}(\sqrt{n})+1$ iterations.
    \begin{enumerate}
        \item If the algorithm terminates by finding an egg cut, then $\|S\|\leq (2+\sqrt{2})\sqrt{n}+4$.
        \item If the algorithm terminates by finding a hitting set, the worst case scenario is that both $H_1$ and $H_2$ are nonempty with the addition of one final vertex.
        Altogether,
        $$\|S\|\leq |H_1|+|H_2|+1 \leq 2\left\lfloor\frac{\Delta}{\sqrt{2n}}\right\rfloor+2\left\lfloor\frac{\Delta}{\sqrt{2n}}\right\rfloor +1 = 2\sqrt{2}\left\lfloor\frac{\Delta}{\sqrt{n}}\right\rfloor +1 \leq 2\sqrt{2n}+1.$$
    \end{enumerate}

    Altogether, we conclude that $\|S\|\leq \max\{(2+\sqrt{2})\sqrt{n}+4, 2\sqrt{2n}+1\} = (2+\sqrt{2})\sqrt{n}+4$ and $\sn(G)\leq (2+\sqrt{2})\sqrt{n}+4$. Since $G$ is a maximal outerplanar graph, by subgraph monotonicity of scramble number, we may conclude that the result holds for all outerplanar graphs.
\end{proof}

\section{Near Outerplanar Graphs} 
\label{Sec:Near Outerplanar}
There exist many notions of what it means for a graph to ``almost" be outerplanar. One such notion are the \emph{nearly outerplanar graphs}\cite{nopgraphs}. A graph is \emph{nearly outerplanar} (NOP) if it is edgeless or has an edge whose deletion results in an outerplanar graph. One well-studied example of NOP graphs are the wheel graphs. The wheel graph $W_m$ on $m+1$ vertices may be defined as $F_m$ with the addition of an edge between $v_1$ and $v_m$. See Figure~\ref{Fig:Wheel} for an example of the wheel graph embedded as coming from the fan graph and embedded in the usual ``wheel" configuration

\begin{figure}[h]
\begin{tikzpicture}[scale=2.0]

\draw [ball color=black] (-3,0) circle (0.55mm);
\draw [ball color=black] (-2,0) circle (0.55mm);
\draw [ball color=black] (-1,0) circle (0.55mm);
\draw [ball color=black] (0,0) circle (0.55mm);
\draw [ball color=black] (-1.5,1.25) circle (0.55mm);
\draw (-3,0)--(-2,0);
\draw (-2,0)--(-1,0);
\draw (-1,0)--(0,0);
\draw (-1.5,1.25)--(-3,0);
\draw (-1.5,1.25)--(-2,0);
\draw (-1.5,1.25)--(-1,0);
\draw (-1.5,1.25)--(0,0);
\draw (-1.5,1) node {{\small $v_0$}};
\draw (-3,-0.25) node {{\small $v_1$}};
\draw (-2,-0.25) node {{\small $v_2$}};
\draw (-1,-0.25) node {{\small $v_3$}};
\draw (0,-0.25) node {{\small $v_4$}};

\draw [<->] (-3:0)  arc (0:180:1.5);

\draw [ball color=black] (1,2) circle (0.55mm);
\draw [ball color=black] (1,0) circle (0.55mm);
\draw [ball color=black] (3,0) circle (0.55mm);
\draw [ball color=black] (3,2) circle (0.55mm);
\draw [ball color=black] (2,1) circle (0.55mm);
\draw (1,2)--(1,0);
\draw (1,0)--(3,0);
\draw (3,0)--(3,2);
\draw (2,1)--(1,2);
\draw (2,1)--(1,0);
\draw (2,1)--(3,0);
\draw (2,1)--(3,2);
\draw (1,2)--(3,2);
\draw (2,1.25) node {{\small $v_0$}};
\draw (0.75,2) node {{\small $v_1$}};
\draw (1,-0.25) node {{\small $v_2$}};
\draw (3,-0.25) node {{\small $v_3$}};
\draw (3.25,2) node {{\small $v_4$}};

\end{tikzpicture}
\caption{The wheel graph $W_4$. On the left $W_4$ is depicted $\mathcal{F}_4$ with an additional edge. On the right $W_4$ is embedded as the usual center vertex with spokes towards the ``outer" vertices}
\label{Fig:Wheel}
\end{figure}
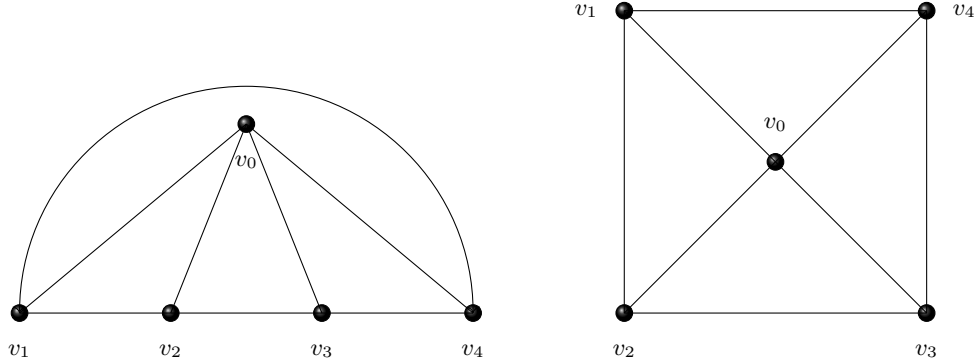
In \cite{DEM21}, the authors showed that, like the fan graph, the wheel graph has a gonality of roughly $2\sqrt{n}$.
\begin{theorem}[\cite{DEM21} Theorem $5.1$]
    For all $m\geq 3$, we have
    $$\gon(W_m) = \lceil\sqrt{m}\rceil-1+\left\lceil \frac{m}{\lceil\sqrt{m}\rceil}\right\rceil$$
\end{theorem}

In order to extend our result to near outerplanar graphs, we use the following result.
\begin{lemma}\label{lem: delete edge}
    Let $G$ be an arbitrary graph. If
    $G'$ can be obtained from $G$ via the deletion of an edge $(u,v)$, then
    $$\sn(G') \leq \sn(G) \leq \sn(G')+1.$$
\end{lemma}
\begin{proof}
    By subgraph monotonicity of scramble number, we obtain that $\sn(G')\leq \sn(G).$ Let $S$ be an arbitrary scramble on $G$.  We define $S' = \{\mathcal{E}\in S\mid \mathcal{E}\setminus$ to be the collection of eggs in $S$ that remain connected once removing the edge $(u,v)$. If $S' =\emptyset$, then all eggs intersect $u$ and $v$. This implies that $h(S) =1 \leq \sn(G')$.

    As such, we consider the case where $S'$ is nonempty and a scramble over $G'$. We may now proceed by cases 
    \begin{enumerate}
        \item If $h(S')\leq \sn(G')$ as a scramble on $G'$. This implies that we have a hitting set $H$ of size at most $\sn(G')$ on $S'$. Since all eggs in $S\setminus S'$ intersect $v$, then $H\cup \{v\}$ is a hitting set for $S$ and $\|S\|\leq \sn(G')+1.$

        \item If $e(S')\leq \sn(G')$ as a scramble on $G'$. This implies that in $G'$ there exists a cut $C$  such that $|C|\leq \sn(G')$ and $C$ separates two eggs $\mathcal{E}_1$ and $\mathcal{E}_2$ in $S'$. This implies that in $G$ $C\cup\{(u,v)\}$ is a cut which separates $\mathcal{E}_1$ and $\mathcal{E}_2$. Thus $\|S\| \leq \sn(G')+1$.
    \end{enumerate}
    Since $S$ was arbitrary, we conclude that $\sn(G)\leq \sn(G')+1.$
\end{proof}
By an inductive argument on Lemma \ref{lem: delete edge} we obtain the following.
\begin{corollary}\label{coll: k delete}
    Let $G$ be an arbitrary graph. If
    $G'$ can be obtained from $G$ via the deletion of $k$ edges, then
    $$\sn(G') \leq \sn(G) \leq \sn(G')+k.$$
\end{corollary}
We may now recreate the proof of Lemma \ref{lem: fan gon} to compute the scramble number of the wheel graph.
\begin{lemma}
    Let $W_m$ be the wheel graph on $m+1$ vertices. If $m = k^2+c$ for some $k\in \Z_{>0}$ with $0\leq c \leq 2k$, then
    $$\sn(W_m) = \left\lceil\sqrt{m}\right\rceil =  \begin{cases}
        \lfloor\sqrt{m}\rfloor+1 & \text{ if }c=0\\
        \lfloor\sqrt{m}\rfloor+2 & \text{ otherwise }
    \end{cases}.$$
\end{lemma}
\begin{proof}
    By Lemma \ref{lem: fan gon} and Lemma \ref{lem: delete edge}, we have that $$\lfloor\sqrt{m}\rfloor+1\leq \sn(W_m) \leq 
        \lfloor\sqrt{m}\rfloor+2  \text{ if }0\leq c\leq k+1$$
        and
    $$\lfloor\sqrt{m}\rfloor+2\leq \sn(W_m) \leq 
        \lfloor\sqrt{m}\rfloor+3  \text{ if } c> k+1.$$    
    This implies that for $1\leq c \leq k+1$ it suffices to build a scramble $S$ such that $\|S\| = k+2$. 
    We define $S$ to be the scramble with $v_0$ and all paths of size $k$ on the exterior cycle. Since $v_0$ is connected to all other vertices and all eggs satisfy $|\mathcal{E}|\leq k$, then $e(S) = k+2$. Now note that, to hit all disjoint paths and $v_0$, we need at least $k+1$ vertices. Since $c>0$, by the pigeonhole principle, we need at least one more vertex to form a hitting set. As such, $h(S)\geq k+2$ and $\|S\| = k+2$.

    For $c=0$ and $c>k+1$, we recreate the proof of Lemma \ref{lem: fan gon} and produce a set $H$ which forces either a hitting set or an egg cut of size at most $k+2$. For $c= 0$ we use $H = \{v_{jk}\mid 0\leq j\leq k\}$ with $|H| = k+1$ and, for $c>k+1$ we use $H =\{v_0\} \cup \{v_{j(k+1)+1}\mid 0\leq j\leq k\}$ with $|H| = k+2$.
\end{proof}

As such, there is a gap between the scramble number and the gonality of $W_n$. Altogether, we may now define a graph to be \emph{$k$-nearly outerplanar} if it is edgeless or has a subset of at most $k$ edges whose deletion results in an outerplanar graph. Then combining Theorem \ref{thm: main} and Corollary \ref{coll: k delete} we may obtain a bound on the scramble number of $k$-nearly outerplanar graphs.
\begin{corollary}\label{coll: k-NOP}
    Let $G$ be a $k$-nearly outerplanar graph on $n$ vertices.  
    $$\sn(G)\leq (2+\sqrt{2})\sqrt{n}+4+k.$$
\end{corollary}

\section{Open Questions}
\label{Sec: Open Questions}
We remark that the bound proven in Theorem \ref{thm: main} has a leading term of $(2+\sqrt{2})\sqrt{n}$. This constant $2+\sqrt{2}$ is a result of balancing both the number of vertices selected for a potential hitting set and the size of the cuts to separate the components defined by those vertices. However, the planar graphs whose scramble number has been studied have a scramble number bounded by roughly $\sqrt{n}$. We may then ask wether the bound can be improved or if there exist a family of outerplanar graphs whose scramble number meets the bound in Theorem \ref{thm: main}.

The question of whether the scramble number of planar graphs is bounded by $O(\sqrt{n})$ remains open. However, Corollary \ref{coll: k-NOP} implies that any counterexample to the conjecture must be ``far" from being outerplanar. 

We also remark that the upper bound on the scramble number of the fan graph and the wheel graph may be obtained by a graph invariant known as \emph{screewidth} \cite{CFGMMMORW}. However, screewidth was not the approach used in the proof of the main result of the paper. As a consequence, we obtain the following open question:
\begin{question}
    Is the screewidth of outerplanar graphs bounded by $O(\sqrt{n})$?
\end{question}
Scramble number is studied primarily because of its relation to graph gonality. Hence, we may also ask the following:
\begin{question}
    Is the gonality of outerplanar graphs bounded by $O(\sqrt{n})$?
\end{question}

\bibliography{references}

\end{document}